\documentclass[a4paper, 10pt, twoside]{article}

\usepackage{amssymb,amsmath,amsfonts,amsthm,graphicx}

 \newtheorem{thm}{Theorem}[section]
 \newtheorem{cor}{Corollary}[section]
 
 \newtheorem{prop}{Proposition}[section]
 
 \newtheorem{lemma}{Lemma}[section]
 \theoremstyle{definition}
 \newtheorem{defi}{Definition}[section]
 \theoremstyle{remark}
 \newtheorem{remark}{Remark}

 \newcommand{\bs}[1]{\boldsymbol{#1}}
 \newcommand{\bl}{{\bs{\lambda}}}
 \newcommand{\C}{\mathbb{C}}
 \newcommand{\D}{\mathcal{D}}
 \newcommand{\E}{\mathbb{E}}
 
 \newcommand{\etal}{{\it et al}.\ }
 
 \newcommand{\K}{\mathcal{K}}
 \newcommand{\N}{\mathbb{N}}
 \newcommand{\R}{\mathbb{R}}
 \newcommand{\T}{\mathbb{T}}
 \newcommand{\X}{\bs{X}}
 \newcommand{\W}{\mathcal{W}}
 
 \newcommand{\fim}{\begin{flushright}
 \vspace{-.5cm}$\diamondsuit$\end{flushright} }

 \renewcommand{\geq}{\geqslant} 
 \renewcommand{\leq}{\leqslant} 

 \numberwithin{equation}{section}
 \numberwithin{table}{section}

 \newcounter{ex} \counterwithin{ex}{section}
 \newenvironment{example}[1]{%
 \par\vspace{\baselineskip}
 \refstepcounter{ex}\noindent
 \textbf{Example~\theex.}}

 \long\def\symbolfootnote[#1]#2{\begingroup%
 \def\thefootnote{\fnsymbol{footnote}}\footnote[#1]{#2}\endgroup}
\begin{document}

\thispagestyle{empty}

 \vskip2cm
 {\centering
 \Large{\bf Pentadiagonal Matrices and an Application to the Centered
MA(1) Stationary Gaussian
Process}\\
 }
 \vspace{1.0cm}
 \centerline{\large{\bf{Maicon J.\ Karling,
Artur O.\ Lopes and S\'{i}lvia R.C.\ Lopes \symbolfootnote[3]{Corresponding
author. E-mail: silviarc.lopes@gmail.com}}}}

 \vspace{0.5cm}

 \centerline{Mathematics and Statistics Institute}
 \centerline{Federal University of Rio Grande do Sul}
 \centerline{Porto Alegre, RS, Brazil}

 \vspace{0.5cm}

 \centerline{\today}

 \begin{abstract} \noindent
In this work, we study the properties of a pentadiagonal symmetric matrix
with perturbed corners. More specifically, we present explicit expressions for
characterizing when this matrix is non-negative and positive definite in two
special and important cases. We also give
a closed expression for the determinant of such matrices. Previous works
present the determinant in a recurrence form but not in an explicit one.  As an
application of these results, we also study the
limiting cumulant generating function associated to the bivariate sequence of
random vectors $\left(n^{-1} (\sum_{k=1}^n X_k^2\,,\ \sum_{k=2}^n X_k
X_{k-1})\right)_{n \in \N}$, when $(X_n)_{n \in \N}$ is the centered
stationary moving average process of first order
with Gaussian innovations. We exhibit the explicit expression of this
limiting cumulant generating function. Finally, we present three examples
illustrating the techniques studied here.

\vspace{2mm} \noindent

\vspace{2mm} \noindent
\end{abstract}


\section{Introduction} 

Pentadiagonal matrices have been explored in many possible ways in recent
decades, most of them for the symmetric case (sometimes, assuming that the symmetric matrix is Toeplitz). Some results address the analysis of its eigenvalues (see Elouafi \cite{elouafi11} and Fasino \cite{fasino88}), others  focus on explicit formulas for its determinant (see Elouafi \cite{elouafi13, elouafi18}, Jia \etal \cite{jia16}, Marr and
Vineyard \cite{marr88} and Solary \cite{solary20}). Other authors examine faster
algorithms for computing the determinant of such matrices (see Cinkir
\cite{cinkir12} and Sogabe \cite{sogabe08}), its use in solving systems of linear
equations (see Jia \etal \cite{jia12}, McNally \cite{mcnally10} and
Nemani \cite{nemani10}), and in the search of explicit formulas for the inverse matrix
(see Wang \etal \cite{wang15} and Zhao and Huang \cite{zhao08}).

However, there are not many works dedicated to the case of pentadiagonal matrices with perturbed corners; to be defined below.

A pentadiagonal matrix is described in the literature as having zeros
everywhere except in its five principal diagonals. In the present work, we
shall consider the following pentadiagonal matrix with perturbed corners
   \begin{equation}\label{matrizdn}
    D_n = \left[%
    \begin{array}{cccccc}
     r & q & s & 0 & \cdots & 0\\
     q & p & q & s & \ddots & \vdots \\
     s & q & p & \ddots & \ddots & 0 \\
     0 & s & \ddots & \ddots & q & s\\
     \vdots & \ddots & \ddots & q & p & q \\
     0 & \cdots & 0 & s & q & r\\
    \end{array}\right].
   \end{equation}
Our purpose with this study is to present few properties of the matrix
$D_n$, with relation to its determinant and positive and non-negative
definiteness. Working around with the matrix $D_n$ is non-trivial. The
pentadiagonal matrices found in Cinkir \cite{cinkir12}, Elouafi
\cite{elouafi13}, Wang \etal \cite{wang15}, or Jia \etal \cite{jia16} serve as
particular cases from the matrix presented in \eqref{matrizdn}. A more advanced
study is given in Solary \cite{solary20}, where the author presents
computational properties for a pentadiagonal band matrix with perturbed
corners, similar to ours, but the elements are disposed in $N \times N$ blocks
of $m \times m$ matrices in its five main diagonals, with $m, N \in \N$.

As we will show here, a particular case of the pentadiagonal matrix in
\eqref{matrizdn} appears in a problem relating to the centered stationary
moving average process of first order (MA(1)) with Gaussian innovations, defined
by the equation
\begin{align*}
 X_n = \varepsilon_n + \phi\, \varepsilon_{n-1}, \quad \mbox{with } |\phi|<1
\mbox{ and } n \in \N,
\end{align*}
where $(\varepsilon_n)_{n \geq 0}$ is a sequence of independent and identically
distributed (i.i.d.) random variables following a Gaussian distribution with
zero mean and unitary variance ($\varepsilon_n \sim \mathcal{N}(0,1)$, for
all $n
\geq 0$). We are interested in the asymptotics of the bivariate
\emph{normalized cumulant generating function}
\begin{equation*}
  L_n(\bl) = \frac{1}{n}\log \mathbb{E}(\exp(n
\langle(\lambda_1, \lambda_2),\W_n\rangle)) = \frac{1}{n} \log
\left(\mathbb{E}\exp \left[\lambda_1 U_n + \lambda_2 V_n \right]\right), \quad
\mbox{for } \bl = (\lambda_1, \lambda_2) \in \R^2,
\end{equation*}
associated to the random vectors sequence $(\W_n)_{n \geq
2}$, where
\begin{equation}\label{Wn}
 \W_n = n^{-1}(U_n, V_n) = n^{-1}\left(\sum_{k=1}^n X_k^2,\
\sum_{k=2}^n X_k X_{k-1}\right).
\end{equation}
The results we obtain for pentadiagonal matrices
will help us in this direction. The main result in this part of the paper
is to give an explicit expression for the limit $\mathcal{L}(\bl) := \lim_{n
\to
\infty} L_n(\bl)$, when it is well defined. A similar discussion appeared in
Karling \etal
\cite{karling20}, where the authors analyzed the bivariate normalized
cumulant generating function associated with the sequence $(\W_n)_{n \geq
2}$, when $(X_n)_{n \in \N}$ is a centered stationary autoregressive
process of first order with Gaussian innovations. In that work, the treatment of the positive definiteness of a tridiagonal matrix was required.

The  normalized cumulant generating function is of great help for obtaining the moments of a given random vector. We point out that for
the practical  use of this property it is required to have an explicit expression for it.
The analytic expression we obtain for $\mathcal{L}(\cdot)$ is quite complex (see Proposition \ref{propconvLn}) but its partial derivatives can be calculated using the Wolfram Mathematica software.

The present work is organized as follows.  Section \ref{secpos} is dedicated to
obtaining  a closed expression for the domain when $D_n$ is non-negative
definite in the presence of the restriction $r \geq p - s$. Furthermore, we
analyze the special case $r = p-s$ to give the explicit domain for which $D_n$
is a positive definite matrix. In Section \ref{SecDeterminante} we compute the
determinant of the matrix $D_n$ by using a recurrence relation proposed in
Sweet \cite{sweet69}. An application to the MA(1) process is presented in
Section 4, where we analyze the asymptotic behavior of the bivariate normalized
cumulant generating function associated to the sequence $(\W_n)_{n \geq 2}$,
given in \eqref{Wn}, and we provide its limiting function. A few examples to
illustrate the theory in practice are exhibited in Section 5. In
Section \ref{conclusionsection} some conclusions are presented.


 \section{Non-negative and positive definiteness of $\bs{D_n}$}
\label{secpos}

We scrutinize in the following subsections when the matrix $D_n$
in \eqref{matrizdn} is non-negative definite if the
restriction $r \geq p - s$ is considered. In addition to this, a sharper result
can be provided for the positive definiteness of $D_n$ in the special case when
$r = p - s$. Both reasonings rely on the results proved in Fasino
\cite{fasino88} and Solary \cite{solary20}. Despite being well known, we recall
two equivalent definitions of non-negative (positive) definite matrices in the
real symmetric case.

 \begin{defi}
  A real symmetric matrix $M = [m_{i,j}]_{n \times n}$ of order $n \times n$ is
said to be non-negative (positive) definite if (see Gilbert \cite{gilbert91} and
Horn \cite{horn13}):
\begin{enumerate}
 \item the scalar $\bs{x}^T M \bs{x}$ is non-negative (positive) for every
non-zero column vector $\bs{x} \in \R^n$;

 \item the eigenvalues of $M$ are all non-negative (positive).
\end{enumerate}
\end{defi}

\subsection{Case $\bs{r \geq p-s}$}

The approach presented in Fasino \cite{fasino88} yields a nice criterion based
on a second‐order polynomial to determine when $D_n$ in \eqref{matrizdn} is a
non-negative definite matrix. We use this criterion to provide an explicit
expression for the domain which characterizes when $D_n$ is non-negative
definite. It is although necessary to require a priori that
$r \geq p - s$.

\begin{lemma}\label{fasino}
 Let $D_n$ be the pentadiagonal matrix defined in \eqref{matrizdn} with $p \geq
0$. Consider the sets
 \begin{equation}\label{D1234}
 \begin{split}
  & \D_1 = \left\{-\frac{p}{2} \leq s < 0 \,,\ -\frac{1}{2}(p+2s) \leq
q \leq \frac{1}{2}(p+2s)\right\},\\
  & \D_2 = \left\{s = 0 \,, \ p \geq 2|q|\right\},\\[1mm]
  & \D_3 = \left\{0 < s \leq \frac{p}{2} \,,\ -\sqrt{4s(p-2s)}\leq q\leq
\sqrt{4s(p-2s)}  \right\} \mbox{ and}\\
  &\D_4 = \left\{0 < s < \frac{p}{6} \,,\ -\frac{1}{2}(p+2s) \leq q <
-\sqrt{4s(p-2s)} \ \lor \, \sqrt{4s(p-2s)} < q \leq \frac{1}{2}(p+2s)\right\}.
 \end{split}
 \end{equation}
 If $r \geq p - s$ and $p, q, s$ lie inside $\D_1 \cup
\D_2 \cup \D_3 \cup \D_4$, then $D_n$ is non-negative definite for all $n
\in \N$.
\end{lemma}

\begin{proof}
  First we observe that if $p = 0$, then, the only possible case where $D_n$
might be non-negative definite is the trivial one, when $p = q = s = 0$. Thus,
we can assume hereafter that $p > 0$. The remaining of the proof stands on
proposition 5 in Fasino \cite{fasino88}, which states that, given
\begin{equation}\label{polyg}
 g(x) = s x^2 + q x + (p - 2s), \quad \mbox{for } x \in \R,
\end{equation}
the matrix $D_n$ is non-negative definite, for all $n \in \N$, if and
only if $g(x) \geq 0$, for all $x \in [-2,2]$.

We separate our analysis in three cases:
\begin{itemize}
 \item {\it Case $s < 0$:} by hypothesis $p > 0$, hence, it follows that
  $q^2 - 4s(p-2s) > 0$ and the equation $g(x) = 0$ has two real roots, given by
\begin{equation}\label{x12}
 x_1 = \frac{-q -\sqrt{q^2-4s(p-2s)}}{2s} \quad \mbox{and} \quad x_2 =
\frac{-q+\sqrt{q^2-4s(p-2s)}}{2s}.
\end{equation}
For the condition $g(x) \geq 0$ to be true for all $x \in [-2,2]$, we must
have simultaneously $x_2 \leq -2$ and $x_1 \geq 2$. The latter relations are
verified if and only if $p, q, s$ lie inside $\D_1$.

\item {\it Case $s = 0$:} in this case, notice that $D_n$ is a tridiagonal
matrix and that $g(x) = q x + p$. Therefore, if $p \geq 2|q|$, then $g(x) \geq
0$ for all $x \in [-2,2]$. Hence, $p, q, s$ must lie inside $\D_2$ for $D_n$
to be non-negative definite.

\item {\it Case $s > 0$:} here we observe that there are two possibilities.
Either $q^2 - 4s(p-2s) \leq 0$ and $g(x) \geq 0$, for all $x
\in \R$, or either $q^2 - 4s(p-2s) > 0$ and $g(x) = 0$
has two real distinct roots, namely, $x_1$ and $x_2$ given in \eqref{x12}. In
the former case, $p, q, s$ must lie inside $\D_3$. In the later case, $g(x)
\geq
0$, for all $x \in [-2,2]$, if and only if $x_2 \leq
-2$ or $x_1 \geq 2$, which gives us the domain $\D_4$ in \eqref{D1234}.
\end{itemize} \vspace{-6mm} \end{proof}

\begin{remark} Note that, if $p, q, s$ belong to $\D_1 \cup \D_2 \cup \D_3 \cup
\D_4$ and $p \geq 0$, then $r \geq p -
s$ implies that $r \geq 0$.
\end{remark}

\begin{remark}
 When considering proposition 5 in Fasino \cite{fasino88}, the term positive
definite should be read as non-negative definite. Additionally, the same
proposition cannot be proved for positive definite matrices in the strict
positive sense, i.e., by just replacing the condition $g(x) \geq 0$, for all $x
\in [-2,2]$, by $g(x) > 0$, for all $x \in [-2,2]$.
\end{remark}

An illustration of the domain $\D_1 \cup \D_2 \cup \D_3 \cup \D_4$ is given
in Figure \ref{domainD1234}. We note that outside this set it may happen
that $D_n$ is non-negative definite for some $n \in \N$, but this does not
generate a contradiction to the result of Lemma \ref{fasino}. In fact, the
statement of this lemma considers the non-negative definiteness of the matrices
$D_n$ for all $n \in \N$.

\begin{figure}[ht!]
 \begin{center}
   \includegraphics[scale =0.35]{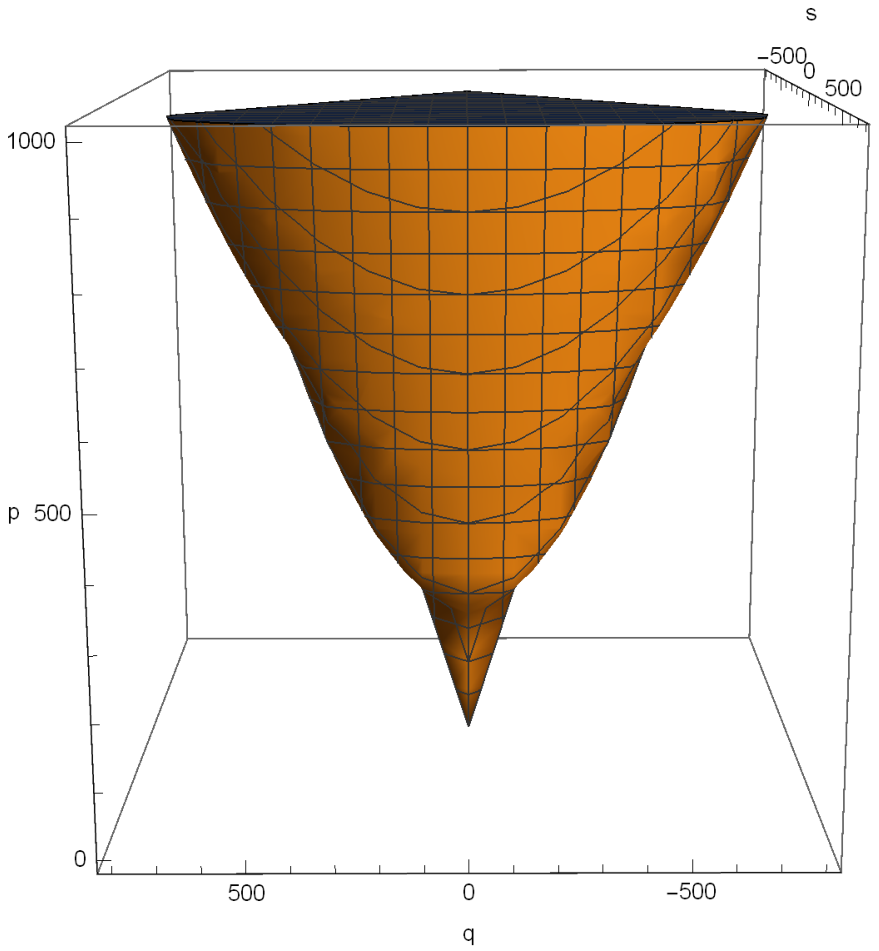} \quad
   \includegraphics[scale =0.35]{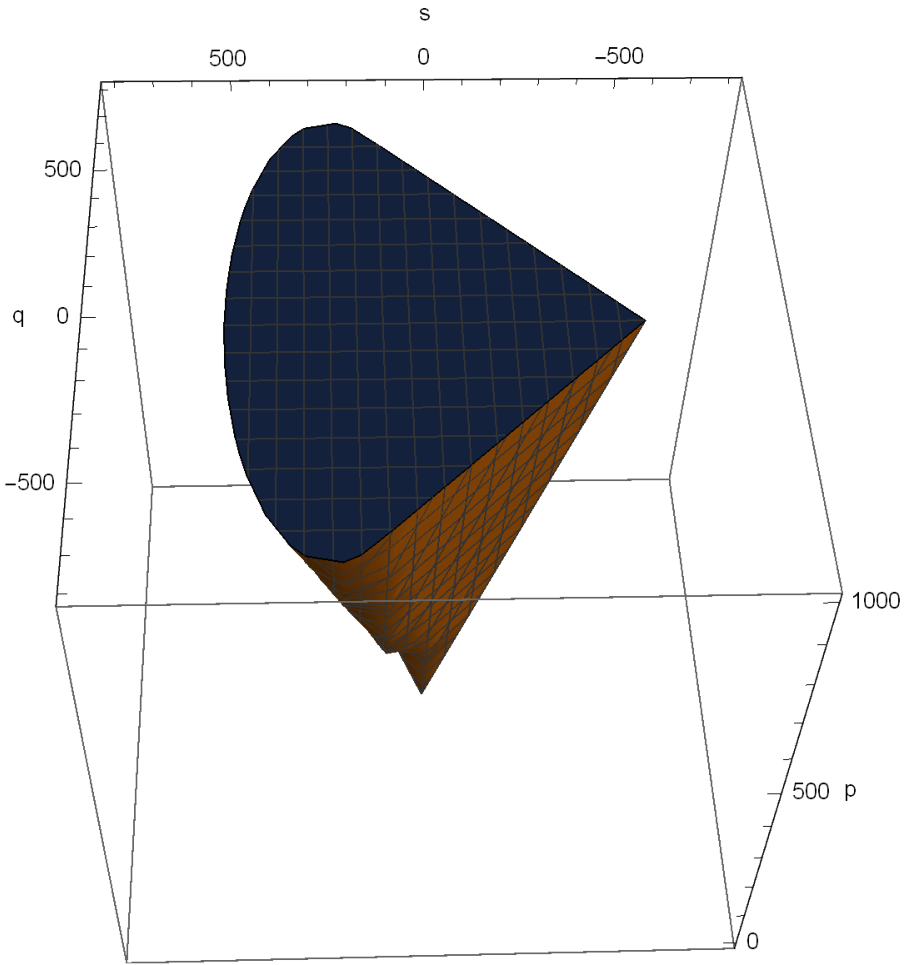} \quad
   \includegraphics[scale =0.35]{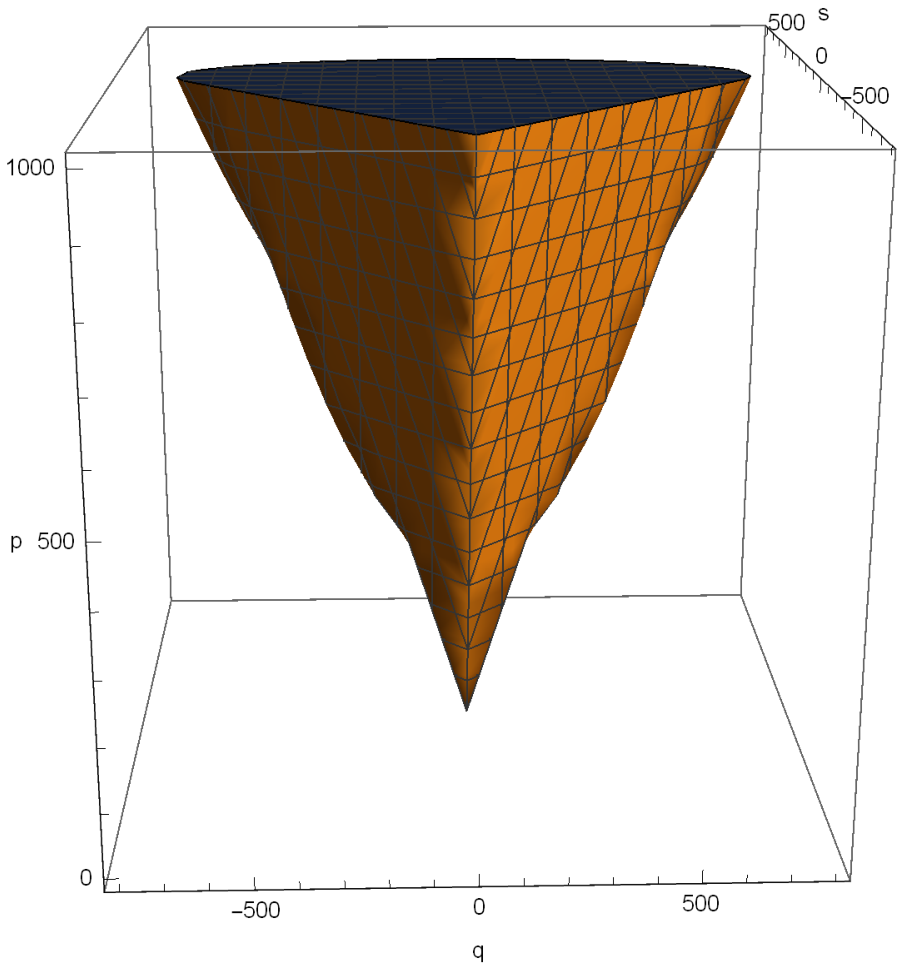}
   \caption{Domain $\D_1 \cup \D_2 \cup \D_3 \cup \D_4$
illustrated, for the cases when $s ,q \in [-800,800]$ and $p
\in [0,1000]$.
\label{domainD1234}}
 \end{center}
\end{figure}

\subsection{Special case $\bs{r = p-s}$}\label{specsec}
It may happen that $r = p-s$ and as a consequence we obtain the
following.

\begin{lemma}\label{eigelema}
If the elements of the matrix $D_n$ in \eqref{matrizdn} satisfy the relation $r
= p-s$, then its eigenvalues are given by
\begin{equation*}
 \alpha_{n,k} = 4 s \cos^2\left(\frac{k \pi}{n+1}\right) +2q \cos\left(\frac{k
\pi}{n+1}\right) +p-2s, \quad \mbox{for } \ 1 \leq k \leq n.
\end{equation*}
\end{lemma}
\begin{proof}
 See theorem 4 in Solary \cite{solary20}.
\end{proof}

Since we have explicitly the general representation for the eigenvalues of
$D_n$
in the special case when $r = p - s$, it is now easy to obtain the determinant of
such matrix. As a consequence from Lemmas \ref{fasino} and \ref{eigelema}, the
following corollary is of
extreme importance.

\begin{cor}\label{nulleigenvalue}
 Let $D_n$ be the matrix in \eqref{matrizdn} with $r = p-s$. Then, it follows
that
\begin{enumerate}
 \item  $D_n$ has a null eigenvalue if and only if
 \begin{equation*}
  4 s \cos^2\left(\frac{k \pi}{n+1}\right) +2q \cos\left(\frac{k
\pi}{n+1}\right) +p-2s = 0,
 \end{equation*}
 for some $k$ such that $1 \leq k \leq n$.
 \item A closed expression for the determinant of $D_n$ is
given by
 \begin{equation*}
  \det(D_n) = \prod_{k=1}^n \left(4 s \cos^2\left(\frac{k \pi}{n+1}\right) +2q
\cos\left(\frac{k \pi}{n+1}\right) +p-2s\right).
 \end{equation*}
 \item Consider
  \begin{equation*}
  \D_0 = \bigcup_{n \in \N} \left\{p,q,s\ \Big|\ 0 < s \,,\ p = 2 s
\left(1+2 \cos^2\left(\frac{k \pi}{n+1}\right)\right)\,,\ q = -4 s
\cos\left(\frac{k \pi}{n+1}\right),\ \mbox{for } k \in \N\right\}.
 \end{equation*}
If $p, q, s$ lie inside $\D_1 \cup \D_2 \cup \D_3 \cup
\D_4 \setminus \D_0$, then $D_n$ is positive definite, for all $n \in \N$.
\end{enumerate}
\end{cor}

\begin{proof}
   By Lemma \ref{eigelema}, the eigenvalues of $D_n$ are given as $\alpha_{n,k}
= 4 s \cos^2\left(\frac{k \pi}{n+1}\right) +2q \cos\left(\frac{k
\pi}{n+1}\right) +p-2s$, for $1 \leq k \leq n$. Hence, statement 1 is
evident. For the proof of statement 2, we note that the determinant of $D_n$ is
equal to the product of its eigenvalues.

Statement 3 is the only one that requires
more caution. In the proof of Lemma \ref{fasino}, we note that inside $\D_1 \cup
\D_2 \cup \D_4$ we have $\alpha_{n,k} > 0$ for all $k, n \in
\N$. Indeed, if $p, q, s \in \D_1 \cup \D_2 \cup \D_4$, the polynomial
$g(\cdot)$, defined in \eqref{polyg}, is non-negative for all $x \in [-2,2]$
and, in the worst scenario, it has a real root at $x = -2$ or $x = 2$.
Since
$\alpha_{n,k} = g\left(2 \cos\left(\frac{k \pi}{n+1}\right)\right)$ and
$\big|\cos\left(\frac{k \pi}{n+1}\right)\big| < 1$, for all $k, n \in \N$,
it follows that $\alpha_{n,k} >
0$, for all $k, n \in \N$. The only
section that $D_n$ can actually have a null eigenvalue is inside
the domain $\D_3$ with $q^2 = 4 s (p-2s)$. In this
case, when $q^2 = 4 s (p-2s)$, it follows that $-q/2s$ is the only root of the
polynomial $g(\cdot)$ and, therefore,
\begin{equation*}
 \alpha_{n,k} = g\left(2 \cos\left(\frac{k \pi}{n+1}\right)\right) = 0 \
\Leftrightarrow \ q = -4 s \cos\left(\frac{k \pi}{n+1}\right).
\end{equation*}
As a solution to the equation $s\left(2 \cos\left(\frac{k
\pi}{n+1}\right)\right)^2 + q \left(2 \cos\left(\frac{k \pi}{n+1}\right)\right)
+ (p-2s) = 0$, we obtain
\begin{equation*}
 p = 2 s\left(1+ 2 \cos^2\left(\frac{k \pi}{n+1}\right)\right).
\end{equation*}
Therefore, the matrix $D_n$, with $r = p - s$, has an eigenvalue equal to zero
if and only if $s > 0$, $q = -4 s \cos\left(\frac{k \pi}{n+1}\right)$ and $ p =
2 s\left(1+ 2 \cos^2\left(\frac{k \pi}{n+1}\right)\right)$.
\end{proof}



\section{An explicit formula for the determinant of the matrix
$\bs{D_n}$}\label{SecDeterminante}

It is possible to find in the literature explicit formulas for the determinant
of pentadiagonal symmetric Toeplitz matrices (see e.g.\ Andeli\'{c} and da
Fonseca \cite{andelic20}, Elouafi \cite{elouafi11, elouafi13}, and Jia \etal
\cite{jia16}). However, little has been done concerning  pentadiagonal
symmetric matrices with perturbed corners. Recently, Solary \cite{solary20}
proposed a closed expression for the determinant and computational
properties for a pentadiagonal matrix disposed by blocks, where the corners in
the main diagonal are perturbed. This matrix by blocks serves as a
generalization of the matrix $D_n$ in \eqref{matrizdn} and its determinant can
be computed from equation (22) in Solary \cite{solary20}. The formula of the
determinant was given with the help of the Sherman-Morrison-Woodbury
formula.

In the present section, we show a closed expression for
the determinant of the matrices $D_n$ and $E_n$, defined in \eqref{matrizen}, by
considering a recursive relation
proposed in Sweet \cite{sweet69}. We also show the explicit expressions for some cases not covered by this author (see Lemmas 3.2 for matrices $D_n$ and $E_n$ and Lemma 3.3-3.5
for the matrix $E_n$). In Theorem 3.1 we exhibit a closed expression for the determinant of the matrix $D_n$, based on the results of Lemmas 3.1-3.5. As far as we know, this explicit expression is totally new and it provides a quicker and efficient way
to compute the determinant of $D_n$. To achieve such aim, we shall consider the sub-matrix
   \begin{equation}\label{matrizen}
    E_n = \left[%
    \begin{array}{cccccc}
     r & q & s & 0 & \cdots & 0 \\
     q & p & q & s & \ddots & \vdots  \\
     s & q & p & \ddots & \ddots & 0 \\
     0 & s & \ddots & \ddots & q & s \\
     \vdots & \ddots & \ddots & q & p & q\\
     0 & \cdots & 0 & s & q & p
     \end{array}\right].
    \end{equation}
Let us denote the determinants of $D_n$ and $E_n$ by $d_n$
and $e_n$, respectively. The recursive relation presented in Sweet
\cite{sweet69} gives us
the following lemma.

\begin{lemma}\label{recrela}
For $n \geq 6$ and $q \neq 0$, the following recursive relations hold
\begin{align}
 &d_{n} = (r - s)\,e_{n-1} + (p\,s - q^2)\,(e_{n-2}-s\,e_{n-3}) + s^3\,(s -
p)\,e_{n-4} + s^5\,e_{n-5}, \label{dn}\\[2mm]
 &e_{n} = (p - s)\,e_{n-1} + (p\,s - q^2)\,(e_{n-2} -s\,e_{n-3}) + s^3\,(s
- p)\,e_{n-4} + s^5\,e_{n-5},\label{en}
\end{align}
with the initial conditions
\begin{align}
 &e_1 = r, \nonumber\\[2mm]
 &e_2 = p\,r-q^2, \nonumber\\[2mm]
 &e_3 = p^2 r - q^2 (r - 2 s) - p\,(q^2 + s^2),\label{intcond} \\[2mm]
 &e_4 = p^3 r-p^2 \left(q^2+s^2\right)-p \left(2 q^2 (r-s)+r
s^2\right)+q^4+2 q^2 s (r-s)+s^4, \nonumber\\[2mm]
 &e_5 = p^4 r + q^4 (r - 4 s) + r s^4 + 2 q^2 s^2 (-r + s) - p^3 (q^2 +
s^2) + p (2 q^4 + 4 q^2 r s + s^4) \nonumber\\
  & \qquad + p^2 (-2 r s^2 + q^2 (-3 r + 2 s)). \nonumber
\end{align}
\end{lemma}
 \begin{proof}
 Immediate from equations (1), (5) and (11) in Sweet \cite{sweet69}.
\end{proof}

\medskip

\begin{remark}
 The initial conditions $e_1, e_2, e_3, e_4, e_5$ in \eqref{intcond} are
defined as the first, second, third, fourth and fifth principal minor of $E_n$,
respectively.
\end{remark}

The case when $q = 0$ is not covered by Sweet's \cite{sweet69} recurrence
relations, but it is not difficult to prove the following.

\begin{lemma}\label{lemaq0}
 For $n \geq 5$ and $q = 0$, the following recursive relations hold
\begin{align}
 &d_{n} = r\,e_{n-1} - p\,s^2\,e_{n-3} + s^4\,e_{n-4}, \label{qdn}\\
 &e_{n} = p\,e_{n-1} - p\,s^2\,e_{n-3} + s^4\,e_{n-4},\label{qen}
\end{align}
with the initial conditions
\begin{align*}
 e_1 = r, \qquad
 e_2 = p\,r, \qquad
 e_3 = p\,(p\,r - s^2), \qquad
 e_4 = (p^2-s^2)(p\,r - s^2).
\end{align*}
\end{lemma}
\begin{proof}
 The proof follows by the induction principle.
\end{proof}

From \eqref{en}, we obtain the following lemma.

\begin{lemma}\label{l22}
 If $q \neq 0$ and $s \neq 0$, then $e_n = \det(E_n)$ may be given by
 \begin{equation}
  e_n = \sum_{j=1}^5 \kappa_j \, \mu_j^n, \label{enclosed1}
 \end{equation}
 where $\mu_1, \cdots, \mu_5$ are given in \eqref{roots}. The coefficients
$\kappa_1, \cdots, \kappa_5$ are described in the following way:
 \begin{enumerate}
  \item if $q^2 \notin \left\{4s (p-2s),\, (p+2s)^2/4\right\}$, then
it holds \eqref{k15};
  \item if $q^2 = 4s (p-2s)$ and $p > 6s$, then it holds \eqref{p>6s};
  \item if $q^2 = 4s (p-2s)$ and $p < 6s$, then it holds \eqref{p<6s};
  \item if $q^2 = (p+2s)^2/4$ and $p \neq 6s$, then it holds \eqref{k15a};
  \item if $q^2 \in \left\{4s(p-2s),\, (p+2s)^2/4 \right\}$ and $p = 6s$, then
it holds \eqref{k15b}.
 \end{enumerate}
\end{lemma}

\begin{proof}
  The result follows by applying the characteristic roots technique to the
associated auxiliary polynomial
  \begin{equation*}
   \rho(z) = z^5 - (p-s)\,z^4 - (p\,s - q^2)\,(z^3 - s\,z^2) - s^3\,(s -
p)\,z - s^5, \qquad \mbox{for } z \in \mathbb{C}.
  \end{equation*}
  The roots of $\rho(\cdot)$ are given by
  \begin{equation}\label{roots}
   \begin{split}
   &\mu_1 = \frac{p - 2s - \alpha - \beta_1}{4},\qquad
    \mu_2 = \frac{p - 2s - \alpha + \beta_1}{4}, \qquad
    \mu_3 = \frac{p - 2s + \alpha - \beta_2}{4}, \\
   &\hspace{3cm}
    \mu_4 = \frac{p - 2s + \alpha + \beta_2}{4} \qquad \mbox{and} \qquad
    \mu_5 = s,
   \end{split}
  \end{equation}
  with
  \begin{align}
   \alpha = \sqrt{(p+2s)^2 - 4 q^2},\ \
   \beta_1 = \sqrt{2(p-2s)(p+2s-\alpha)-4q^2} \ \ \mbox{and} \ \
   \beta_2 = \sqrt{2(p-2s)(p+2s+\alpha)-4q^2}. \nonumber
  \end{align}
 Let us separate the proof in four cases.

 \noindent {\it Case 1:} if $q^2 \notin \left\{4 s\,(p-2s),\,
(p+2s)^2/4\right\}$, then $\alpha$, $\beta_1$ and $\beta_2$ are non-zero, and as
a consequence, $\mu_1, \cdots, \mu_5$ are distinct roots of the polynomial
$\rho(\cdot)$. Thus, each solution to the recurrence in \eqref{en} is of the
form \eqref{enclosed1}, where the coefficients $\kappa_j$, for $j =
1,\cdots,5$, are
the solution to
the 5-by-5 Vandermonde linear system
  \begin{align*}
   \left[%
    \begin{array}{ccccc}
     1 & 1 & 1 & 1 & 1 \\
     \mu_1   & \mu_2   & \mu_3   & \mu_4   & \mu_5 \\[1mm]
     \mu_1^2 & \mu_2^2 & \mu_3^2 & \mu_4^2 & \mu_5^2 \\[1mm]
     \mu_1^3 & \mu_2^3 & \mu_3^3 & \mu_4^3 & \mu_5^3 \\[1mm]
     \mu_1^4 & \mu_2^4 & \mu_3^4 & \mu_4^4 & \mu_5^4
    \end{array}\right]
    \left[%
    \begin{array}{c}
     \kappa_1'\\[0.8mm]
     \kappa_2'\\[0.8mm]
     \kappa_3'\\[0.8mm]
     \kappa_4'\\[0.8mm]
     \kappa_5'
    \end{array}\right] =
    \left[%
    \begin{array}{c}
     e_1\\[0.8mm]
     e_2\\[0.8mm]
     e_3\\[0.8mm]
     e_4\\[0.8mm]
     e_5
    \end{array}\right]
  \end{align*}
with $\kappa_j' = \kappa_j\,\mu_j$ and $e_j$ representing the initial
conditions given in \eqref{intcond}, for $j = 1, \cdots, 5$. We used the
Wolfram Mathematica software (version 11.2) to find these coefficients,
obtaining the expressions:
\begin{align}
 \kappa_1 = \K&(-\alpha, \beta_2, -\beta_1), \qquad
 \kappa_2 = \K(-\alpha, \beta_2,  \beta_1),  \qquad
 \kappa_3 = \K( \alpha, \beta_1, -\beta_2),  \nonumber \\
&\kappa_4 = \K( \alpha, \beta_1,  \beta_2) \qquad  \mbox{and} \qquad
 \kappa_5 = \frac{2 s\,(r+s-p)}{q^2-4 s (p-2 s)}, \label{k15}
\end{align}
where
\begin{equation}\label{Kxyz}
 \K(x,y,z) = \frac{ 64
 \begin{pmatrix}
      2 s^4 (2 s + 3p + x + z) + p s^2 (4 q^2 - 2 s (p-x) -(p+x)(2p+z)) \\
       -\,2 s q^2 (4q^2-2p^2 - (p-2s)(2x+z)- x\,z)\\
  +\,r  \begin{pmatrix}
       4 q^4 - q^2 (p+x) (2p+z) +2 s q^2 (p - 2s+ 3x +2z)\\
       - 2 s^2 p\,(p+x+z) - 2 s^3 (p-2s-x+z)
      \end{pmatrix} \\
  +\,(q^2 - p r) \begin{pmatrix}
                2 s^2 (2s + 3 p + x) + 2 q^2 (3p - 4s + x + z)\\
                - (s\,  (p-x) + p\,(p+x))(2 p + z)
               \end{pmatrix}
 \end{pmatrix}}{z\,(p-2s+x+z)\,(p-6s+x+z)\,((2x+z)^2-y^2)},
\end{equation}
for $x, y, z \in \C$. We note that the coefficients $\kappa_1, \cdots,
\kappa_5$ in \eqref{k15}-\eqref{Kxyz} are not well defined when $q^2 \in
\left\{4 p\,(p-2s),\, (p+2s)^2/4\right\}$. In these cases, some of the roots
$\mu_1, \cdots, \mu_5$ have multiplicity greater than 1. Thus, the solution to
the recurrence in \eqref{en} takes another form and the coefficients might
depend on $n$.

\noindent {\it Case 2:} if $q^2 = 4 s\,(p-2s)$, then $\alpha = |p -
6s|$. Let us consider $\gamma = \sqrt{(p-6 s) (p-2 s)}$. On the one hand, if $p
> 6s$, we get $\beta_1 = 0$ and $\beta_2 = 2 \gamma$, implying that $\mu_1 =
\mu_2 = \mu_5 = s$, $\mu_3 = (p-4s-\gamma)/2$ and $\mu_4 = (p-4s+\gamma)/2$. It
follows that \eqref{enclosed1} is a
solution to the recurrence in \eqref{en}, with
\begin{equation}\label{p>6s}
 \kappa_1 = \K_1, \quad \kappa_2 = 2\,n\ \K_2(2), \quad \kappa_3 =
\K_3(\gamma), \quad \kappa_4 = \K_3(-\gamma) \quad \mbox{and} \quad
\kappa_5 = n^2 \ \K_2(1),
\end{equation}
 where
\begin{equation} \label{K12}
  \K_1 = \frac{p^2-p\,(r+8 s)+2s\,(2 r+11 s)}{(p-6 s)^2}, \qquad
  \K_2(j) = \frac{p-r- j\,s}{p-6 s}, \quad \mbox{for } j = 1,2,
\end{equation}
and
\begin{equation} \label{K3}
\K_3(z) = \frac{2\,s^2\,(p-2s)^2\,(p-4s+z) \left((r-p) (p-4
s + z)+2 s^2\right)}{z\,(4s\,(3s-z)+ p\,(p-8 s+z))^3}, \quad \mbox{for } z \in
\C.
\end{equation}
Note that $\kappa_2$ and $\kappa_5$ are dependent on $n$ and $n^2$,
respectively. On the other hand, if $p < 6s$, we get
$\beta_1 = 2 \gamma$ and $\beta_2 = 0$, implying that $\mu_1 =
(p-4s-\gamma)/2$, $\mu_2 = (p-4s+\gamma)/2$ and
$\mu_3 = \mu_4 = \mu_5 = s$. Then, it follows that \eqref{enclosed1} is a
solution to the recurrence in \eqref{en}, with
\begin{equation}\label{p<6s}
 \kappa_1 = \K_3(\gamma), \quad \kappa_2 = \K_3(-\gamma), \quad
\kappa_3 = \K_1,  \quad \kappa_4 = 2\,n\ \K_2(2) \quad \mbox{and} \quad
\kappa_5 = n^2\ \K_2(1),
\end{equation}
for $\K_1$ and $\K_2(\cdot)$ defined in \eqref{K12} and $\K_3(\cdot)$ defined
in \eqref{K3}. Note that in this case, $\kappa_4$ and $\kappa_5$ are dependent
on $n$ and $n^2$, respectively.

\noindent {\it Case 3:} if $q^2 = (p+2s)^2/4$ and $p \neq 6s$, let us denote
$\delta = \sqrt{(p-6s)(p+2s)}$. Then $\alpha = 0$ and $\beta_1 = \beta_2 =
\delta$, implying that $\mu_1 = \mu_3 = (p-2s-\delta)/4$ and $\mu_2 = \mu_4 =
(p-2s+\delta)/4$, with $\mu_5 = s \neq \mu_j$, for $j = 1, 2, 3, 4$.
The solution of the recurrence in \eqref{en} is given in this case by
\eqref{enclosed1} with
\begin{align}
 \kappa_1 = \K_4(\delta), \quad \kappa_2 = \K_4(-\delta),
\quad \kappa_3 =  n\ \K_5(\delta), \quad \kappa_4 = n \
\K_5(-\delta) \quad \mbox{and} \quad \kappa_5 =
 \frac{8\,s\,(r+s-p)}{(p-6\,s)^2}, \label{k15a}
\end{align}
where
\begin{align*}
&\K_4 (z) = \frac{8 \begin{pmatrix}
p^5\,(12s - p + z) - 2\,p^4 s\, (24 s -4 r + 5 z) - 4\,p^3 s\,(2 r
(10s + z) - s (16s + 9z))\\
+\,8\,p^2 s^2\,(4r (5s + 2z) + s (20s - 7z)) + 16\,p s^3\,(2 r
(8s - 3z) - 5s (8s + z))\\
- 32\,s^4\,(2\,r\,(6s + z) + s\,(6s - 7z))
\end{pmatrix}}{z\,(p-6 s)\,(p - 2s - z)^2\,(p-6s-z)^2}
\end{align*}
and
\begin{align*}
&\K_5(z) = \frac{4\,
\begin{pmatrix}
p^4\,(2r + 4s - p + z) - 2\,p^3\,(r (6 s + z) - s (6 s - z)) -
8\,p^2 s\,(r (s - z) + s (s + z))\\
+\,8\,p\,s^2\, (r (8 s + z) - s (6 s + z)) - 16\,s^3\,\left(4 s^2 - r
(2 s - z)\right)\end{pmatrix}}{z^2\,(p - 6s - z)\,(p-2
s - z)^2},
\end{align*}
for $z \in \C$. Note that $\kappa_3$ and $\kappa_4$ are both dependent on $n$.

\noindent {\it Case 4:} if $q \in \left\{4 s\,(p-2s),\ (p+2s)^2/4\right\}$ and
$p = 6s$, then $\mu_1 = \cdots = \mu_5 = s$ and the solution of the recurrence in
\eqref{en} is given by \eqref{enclosed1} with
\begin{equation}\label{k15b}
 \kappa_1 = 1, \ \kappa_2 = n \left(\frac{r+8 s}{6 s}\right), \
\kappa_3 = n^2 \left(\frac{5 r-7s}{12s}\right), \ \kappa_4 =
n^3 \left(\frac{r-4 s}{3 s}\right) \ \mbox{and} \ \kappa_5 =
n^4 \left(\frac{r-5 s}{12 s}\right).
\end{equation}
Note that $\kappa_j$ depends on $n^{j-1}$, for $j = 2, 3, 4, 5$.
\end{proof}

An analogous result follows when $q = 0$.

\begin{lemma}\label{lq0}
 If $q = 0$ and $s \neq 0$, then $e_n = \det(E_n)$ may be given by
 \begin{equation} \label{enclosed2}
  e_n = \sum_{j=1}^4 \kappa_j \, \nu_j^n,
 \end{equation}
 where
  \begin{equation*}
   \nu_1 = -s, \quad \nu_2 = s, \quad
   \nu_3 = \frac{1}{2} \left(p-\sqrt{p^2-4 s^2}\right) \quad \mbox{and} \quad
\nu_4 = \frac{1}{2} \left(p+\sqrt{p^2-4 s^2}\right).
  \end{equation*}
 The coefficients $\kappa_1, \cdots, \kappa_4$ are described in the following
 way:
\begin{enumerate}
 \item if $p^2 \neq 4s^2$, then $\kappa_1 = \K_6(1)$, $\kappa_2 =
\K_6(-1)$, $\kappa_3 = \K_7(1)$ and $\kappa_4 = \K_7(-1)$, where
 \begin{align*}
 \K_6(j) = \frac{p-r+j\,s}{2 (p +2\,j\,s)} \quad \mbox{and} \quad
 \K_7(j) = \frac{p^3 r -p^2 s^2-3\,p\,r s^2+2 s^4 + j \left(p s^2 + r s^2
-p^2 r\right) \sqrt{p^2-4 s^2} }{\left(p^2-4 s^2\right) \left(p^2-2 s^2 -j\,p\,
\sqrt{p^2-4 s^2}\right)},
\end{align*}
for $j = -1,1$;
 \item if $p = 2s$, then
 \begin{align*}
 \kappa_1 = \frac{3 s-r}{8 s},\quad \kappa_2 = \frac{r+5 s}{8 s},\quad \kappa_3
= n \left(\frac{r}{2 s}\right) \quad \mbox{and} \quad \kappa_4 = n^2
\left(\frac{r-s}{4 s}\right);
\end{align*}
 \item if $p = -2s$, then
 \begin{align*}
 \kappa_1 = \frac{5 s-r}{8 s},\quad \kappa_2 = \frac{r+3 s}{8 s}, \quad
\kappa_3 = -n\left(\frac{r}{2 s}\right) \quad \mbox{and} \quad
\kappa_4 = - n^2 \left(\frac{r+s}{4 s}\right).
\end{align*}
\end{enumerate}
Note that, if $p = \pm 2s$, then $\kappa_3$ and $\kappa_4$ depend on $n$ and
$n^2$, respectively.
\end{lemma}

\begin{proof}
The proof is similar to the one of Lemma \ref{l22}.
\end{proof}

In the case when $s = 0$ we get the following lemma.

\begin{lemma}\label{lemas0}
 If $s = 0$ then $E_n$ in \eqref{matrizen} is a tridiagonal matrix with
determinant equal to
\begin{equation}\label{enclosed3}
 e_n =
 \begin{cases}
 \kappa_1 \xi_1^n + \kappa_2 \xi_2^n,& \mbox{if } q \neq 0,\\
 r\,p^{n-1},& \mbox{if } q = 0,
       \end{cases}
\end{equation}
where $\xi_1$ and $\xi_2$ are given by \eqref{z12}. The coefficients $\kappa_1,
\kappa_2$ are described in the following way:
\begin{enumerate}
 \item if $p^2 \neq 4 q^2$, then it holds \eqref{k12a};
 \item if $p^2 = 4 q^2$, then it holds \eqref{k12b}.
\end{enumerate}
\end{lemma}

\begin{proof}
 If $s = 0$ then \eqref{en} simplifies to
 \begin{equation}\label{reqs0}
  e_n = p\,e_{n-1} - q^2\,e_{n-2}.
 \end{equation}
In the case when $q = 0$, it follows that $E_n$ is a diagonal matrix with
determinant equal to $\det(E_n) = r\,p^{n-1}$. Whereas if $s \neq
0$, consider
 \begin{equation}\label{z12}
  \xi_1 = \frac{p - \sqrt{p^2 - 4 q^2}}{2} \quad \mbox{and} \quad \xi_2 =
\frac{p + \sqrt{p^2 - 4 q^2}}{2}.
 \end{equation}
 The solutions to the recurrence relation in \eqref{reqs0} are thus given by
 \begin{equation*}
  e_n = \kappa_1 \xi_1^n + \kappa_2 \xi_2^n,
 \end{equation*}
 with
 \begin{equation}\label{k12a}
  \kappa_1 =\frac{p-2r + \sqrt{p^2-4q^2}}{2 \sqrt{p^2-4q^2}} \quad \mbox{and}
\quad \kappa_2 = \frac{2r - p + \sqrt{p^2-4q^2}}{2 \sqrt{p^2-4q^2}}, \quad
\mbox{if} \ p^2 \neq 4 q ^2,
 \end{equation}
 and
 \begin{equation}\label{k12b}
  \kappa_1 = 1 \quad \mbox{and} \quad \kappa_2 = n\left(\frac{2r -
p}{p}\right), \quad \mbox{if} \ p^2 = 4 q^2.
 \end{equation}
\end{proof}

By inserting  the formulas in \eqref{enclosed1} and
\eqref{enclosed3} into the recurrence relation \eqref{dn} and the
formula \eqref{enclosed2} into the recurrence relation \eqref{qdn}, we obtain
an explicit formula for the determinant of $D_n$.

\begin{thm}\label{thedet}
 The determinant of the matrix $D_n$ in \eqref{matrizdn} is given by
 \begin{align}\label{detDn}
  \det D_n =
  \begin{cases}
   \sum_{j=1}^5 \kappa_j \, f(\mu_j) \, \mu_j^{n-5}, & \mbox{if } q \neq 0,\ s
\neq 0 \mbox{ and } n \geq 6,\\
   \sum_{j=1}^4 \kappa_j \, f(\nu_j) \, \nu_j^{n-4}, & \mbox{if } q = 0,\ s
\neq 0 \mbox{ and } n \geq 5,\\
   \sum_{j=1}^2 \kappa_j \, f(\xi_j) \, \xi_j^{n-2}, & \mbox{if } q \neq 0,\ s
= 0 \mbox{ and } n \geq 3,\\
   r^2\,p^{n-2}, & \mbox{if } q = s = 0 \mbox{ and } n \geq 3.
  \end{cases}
 \end{align}
 where $f(\cdot)$ is the polynomial function defined by
 \begin{equation*}
  f(z) =
  \begin{cases}
   (r-s)\,z^4 + (p s -q^2)(z^3-s z^2) + s^3(s-p)\,z + s^5,& \mbox{if } q
\neq 0 \mbox{ and } s \neq 0,\\
   \, r\,z^3 - p\,s^2\,z + s^4,& \mbox{if } q = 0 \mbox{ and } s \neq 0,\\
   \, r\,z - q^2,& \mbox{if } s = 0.
  \end{cases}
 \end{equation*}
\end{thm}
\begin{proof}
 The result in \eqref{detDn} is a consequence of Lemmas
\ref{recrela}--\ref{lemas0}.
\end{proof}

\begin{remark}
 If we consider $D_n$ defined for $n = 3$ and $n = 4$, respectively, as
   \begin{equation}\label{matD34}
    D_3 = \left[%
    \begin{array}{ccc}
     r & q & s \\
     q & p & q \\
     s & q & r
    \end{array}\right]
   \qquad \mbox{and} \qquad
    D_4 = \left[%
    \begin{array}{cccc}
     r & q & s & 0\\
     q & p & q & s\\
     s & q & p & q\\
     0 & s & q & r
    \end{array}\right],
   \end{equation}
 then the expression of the determinant in \eqref{detDn} is true for all
 cases when $n\geq 3$.
\end{remark}


 \section{Application to the centered MA(1) stationary Gaussian process}
 \label{secMA1}

 Consider the stochastic process $(X_n)_{n \in \N}$ defined by the equation
 \begin{equation}\label{ma1}
  X_{n} =  \varepsilon_n + \phi\, \varepsilon_{n-1}, \quad \mbox{with }
|\phi| < 1 \mbox{ and } n \in \N,
 \end{equation}
 where $(\varepsilon_n)_{n \geq 0}$ is a sequence of i.i.d.\ random variables,
with $\varepsilon_n \sim
\mathcal{N}(0,1)$, for each $n \geq 0$. The spectral density
function associated to $(X_n)_{n \in \N}$ is given by
\begin{equation*}
 h_\phi(\omega)= 1 + \phi^2 + 2\,\phi\,\cos(\omega), \quad \mbox{for } \omega
\in \mathbb{T}=[-\pi,\pi).
\end{equation*}
 Since $(X_n)_{n \in \N}$ is stationary (see definition 3.4 in Shumway and
Stoffer \cite{shumway16}), we have by \eqref{ma1} that $X_n \sim
\mathcal{N}(0, 1 +\phi^2)$. Moreover, the hypothesis $|\phi| < 1$ guarantees
that this process is also invertible (see theorem 3.1.2 in Brockwell and Davis
\cite{brockwell91}).

 It is common in natural sources to appear data sets that may be modeled by a
process as the one given in equation \eqref{ma1}. The job of a statistician is
to identify the pattern of these data sets and associate it with such a model.
The process given in \eqref{ma1} is called a moving average process of first
order (MA(1) process). The book by Brockwell and Davis \cite{brockwell91} gives
a full treatment in the subject of MA(1) processes, of which we recall the most
important properties related to it:
\begin{itemize}
 \item if $X$ is a random variable defined on a probability space $(\Omega,
\Sigma, \mathbb{P})$, the expected value is defined by the Lebesgue integral
\begin{equation*}
 \E(X) = \int_{\Omega} X(\omega) \, d\,\mathbb{P}(\omega)
\end{equation*}
and the variance of $X$ is given by $\mbox{Var}(X) = \E(X^2) - \E(X)^2$;

\item the spectral density function of the process $(X_n)_{n \in \N}$ in
\eqref{ma1} satisfies $h_\phi(\omega) = h_\phi(-\omega) >
0$, for all $\omega \in \T$, and $\int_{-\pi}^{\pi} h_\phi(\omega) d \omega <
\infty$;

\item the autocovariance function $\gamma_X(k) = \mathbb{E}(X_{n+k} {X_n}) -
\mathbb{E}(X_{n+k}) \, \mathbb{E}(X_n)$ of $(X_n)_{n \in\N}$ depends on
$h_\phi(\cdot)$ in the sense that
 \begin{equation*}
  \gamma_X(k) = \frac{1}{2\pi}\int_{-\pi}^\pi e^{ik\omega} h_\phi(\omega) d
\omega;
 \end{equation*}

\item the Toeplitz matrix $T_n(h_\phi)$ associated with $h_\phi(\cdot)$
coincides with the autocovariance matrix of the process $(X_n)_{n \in \N}$ and
it is given by
\begin{equation}\label{toepdefi}
 T_n(h_\phi) = \left(\frac{1}{2\pi} \int_{-\pi}^{\pi} e^{i(j-k)\omega}\,
 h_\phi(\omega)\,d\omega \right)_{1\leq j, k \leq n};
\end{equation}

\item the matrix $T_n(h_\phi)$ is symmetric and positive definite.
\end{itemize}

Here we tackle the following problem: let us assume that there is a set of
observations $X_1, \cdots, X_n$ from the process given in \eqref{ma1}. For
$\X_n = (X_1, \ldots, X_n)$ and $\X_n^T$ denoting the transpose of $\X_n$,
consider the random vector
 \begin{equation*}
  \W_n = \frac{1}{n}\left(U_n, V_n\right),
 \end{equation*}
 where
 \begin{equation*}
  U_n = \X_n^T\, T_n(\varphi_1)\, \X_n = \sum_{k=1}^n X_k^2, \qquad V_n =
\X_n^T\, T_n(\varphi_2)\, \X_n = \sum_{k=2}^n X_k X_{k-1},
 \end{equation*}
 and $T_n(\varphi_j)$ being, respectively, the Toeplitz matrices associated
with $\varphi_j(\cdot): \T \rightarrow \R$, for $j = 1, 2$, defined by the
functions
\begin{equation*}
 \varphi_1(\omega) = 1, \qquad \varphi_2(\omega) = \cos(\omega).
\end{equation*}

We are interested in the asymptotic behavior of the normalized cumulant
generating function associated to $\W_n$, defined by
\begin{align*}
 L_n(\bl) &= \frac{1}{n}\log \mathbb{E}(\exp(n
\langle(\lambda_1, \lambda_2),\W_n\rangle)) = \frac{1}{n} \log
\left(\mathbb{E}\exp \left[\lambda_1 U_n + \lambda_2 V_n \right]\right)\\[2mm]
 &= \frac{1}{n} \log \left(\mathbb{E} \exp \left[\X_n^T\left(\lambda_1
T_n(\varphi_1) + \lambda_2 T_n(\varphi_2)\right)\X_n\right] \right),
\end{align*}
for each $\bl = (\lambda_1, \lambda_2) \in \R^2$. From the definition given in
\eqref{toepdefi}, it is easy to show that
linearity holds on Toeplitz matrices. If we set $\varphi_{\bl}(\cdot)
= \lambda_1
\varphi_1(\cdot) + \lambda_2 \varphi_2(\cdot)$, we note that
\begin{equation*}
 L_n(\bl) = \frac{1}{n} \log \left(\mathbb{E}\exp\left[\X_n^T
T_n(\varphi_{\bl}) \X_n\right]\right),
\end{equation*}
with
\begin{equation*}
 T_n(\varphi_{\bl}) = \frac{1}{2} \left(\begin{array}{ccccc}
        2\lambda_1 & \lambda_2 & 0      & \cdots & 0\\
        \lambda_2    & 2 \lambda_1 & \lambda_2 & \ddots & \vdots\\
        0      & \ddots & \ddots & \ddots & 0\\
        \vdots & \ddots & \lambda_2 &  2 \lambda_1 & \lambda_2\\
        0      & \cdots & 0 & \lambda_2 & 2\lambda_1
       \end{array} \right).
\end{equation*}

Since the random vector $\X_n$ follows a $n$-variate Gaussian distribution
and the matrix $T_n(\varphi_\bl)$ is symmetric, as observed in Bercu \etal
\cite{bercu97}, we may rewrite ${\X_n}^T T_n(\varphi_\bl) {\X_n}$ as
\begin{equation*}
 {\X_n}^T\,T_n(\varphi_\bl)\,{\X_n} = \sum_{k=1}^n \alpha_{n,k}^\bl Z_{n,k},
\end{equation*}
where $\{\alpha_{n,k}^\bl\}_{k=1}^n$ are the eigenvalues of
$T_n(\varphi_\bl)\,T_n(h_\phi)$ and $\{Z_{n,k}\}_{k=1}^n$ is a sequence of
i.i.d.\ random variables, each one having a chi-squared distribution with one
degree of freedom. A simple algebraic proof shows that
$\{\alpha_{n,k}^\bl\}_{k=1}^n$ and $\{1 - 2\,\alpha_{n,k}^\bl\}_{k=1}^n$ are
also the eigenvalues of $T_n(h_\phi)^{1/2} \, T_n(\varphi_\bl) \,
T_n(h_\phi)^{1/2}$ and $I_n - 2\,T_n(\varphi_\bl)\,T_n(h_\phi)$,
respectively. Hence, from the independence of the random variables
$\{Z_{n,k}\}_{k=1}^n$, it turns out that $L_n(\bs{\cdot})$ can bee expressed as
(see Karling \etal \cite{karling20}):
\begin{align} \label{normcum}
L_n(\bl) &= \frac{1}{n} \log \left(\mathbb{E}\exp\left[{\X_n}^T\,
T_n(\varphi_\bl)\,{\X_n}\right]\right) =
\begin{cases}
 -\frac{1}{2 n} \sum_{k=1}^n \log (1 - 2 \alpha_{n,k}^\bl),& \mbox{if }
\alpha_{n,k}^\bl < \frac{1}{2},\, \forall\,1 \leq k \leq n,\\
+\infty,& \mbox{otherwise}.
\end{cases}
\end{align}

From \eqref{normcum} we note that the condition $\alpha_{n,k}^\bl <
\frac{1}{2}$, for all $k$ such that $1 \leq k \leq n$, is
the equivalent of requiring that $D_{n,\bl} = I_n -
2\,T_n(\varphi_\bl)\,T_n(h_\phi)$ must be a positive definite matrix, where
\begin{equation} \label{pqrs}
 D_{n,\bl} = \left[%
 \begin{array}{cccccc}
     r & q & s & 0 & \cdots & 0\\
     q & p & q & s & \ddots & \vdots \\
     s & q & p & \ddots & \ddots & 0 \\
     0 & s & \ddots & \ddots & q & s \\
     \vdots & \ddots & \ddots & q & p & q \\
     0 & \cdots & 0 & s & q & r
 \end{array}\right], \hspace{0.6cm} \mbox{with} \hspace{0.6cm}
 \begin{cases}
     r &= 1 - 2 \lambda_1(1+\phi^2) - \lambda_2 \phi,\\
     p &= 1 - 2 \lambda_1(1+\phi^2) - 2 \lambda_2 \phi,\\
     q &= -2\lambda_1 \phi - \lambda_2(1+\phi^2),\\
     s &= -\lambda_2 \phi.
  \end{cases}
 \end{equation}
To avoid confusion, we shall adopt the notation $D_{n,\bl}$ to
distinguish the particular case in \eqref{pqrs} from the general one in
\eqref{matrizdn}, and in the sequel, we say that $D_{n,\bl}$ is the
pentadiagonal matrix associated to the MA(1) process. Thus, it follows that
\begin{align}\label{assynlog}
 L_n(\bl) =
  \begin{cases}
   - \frac{1}{2 n} \log(\det(D_{n,\bl})), & \mbox{if }
D_{n,\bl} \mbox{ is positive definite,}\\
   + \infty, & \mbox{otherwise}.
  \end{cases}
\end{align}
 It remains to check for the convergence of $-(1/2n) \log(\det(D_{n,\bl}))$,
which is given by the next proposition.

\begin{prop}\label{propconvLn}
  Let $\bl = (\lambda_1,\lambda_2) \in \R^2$ and $\D_\bl = \D_\bl^1 \cup
\D_\bl^2$, with $\D_\bl^1$ and $\D_\bl^2$ given in \eqref{D1} and \eqref{D2},
respectively. Then, $\mathcal{L}(\bl) := \lim_{n \to \infty} L_n(\bl) = \lim_{n
\to \infty} -(1/2n) \log(\det(D_{n,\bl}))$, where
  \begin{equation}\label{Llim}
  \mathcal{L}(\bl) =
   \begin{cases}\displaystyle
    -\frac{1}{2} \log\left[\frac{(p - 2s)(1 + \sqrt{1 - A^2})(1 +
\sqrt{1 - B^2})}{4}\right],& \mbox{for }\, \bl \in \D_\bl \setminus
\overline{\D_\bl^0}, \\
    +\infty,& \mbox{otherwise},
   \end{cases}
  \end{equation}
 with
 \begin{equation}\label{AB}
  A = \frac{q - \sqrt{q^2 - 4 s (p -
2 s)}}{p-2s} \quad \mbox{and} \quad B = \frac{q + \sqrt{q^2 - 4 s (p - 2
s)}}{p-2s},
 \end{equation}
$p, q$ and $s$ defined as in \eqref{pqrs}, and
$\overline{\D_\bl^0}$ denotes the closure of $\D_\bl^0$, given in \eqref{D0}.
 \end{prop}

 \begin{proof}
 As $D_{n,\bl}$ is a matrix that satisfies the relation $r = p-s$, Lemma
\ref{fasino} and Corollary \ref{nulleigenvalue} are applicable. The domains
$\D_1, \D_2, \D_3, \D_4$ in \eqref{D1234} can be rewritten in terms of
$\lambda_1, \lambda_2$ and $\phi$, as the union of the two following sets
\begin{align}
    \D_\bl^1= & \left\{\frac{1 + 4\lambda_2 \phi}{2 (1+\phi^2)}
\leq \lambda_1 \leq \frac{1}{2(1+\phi^2)}\,, \ (2\lambda_1 \phi
+ \lambda_2(1+\phi^2))^2 \leq -4 \lambda_2 \phi (1-2\lambda_2 (1+\phi^2))
\right\} \label{D1}
\intertext{and}
    \D_\bl^2 = &\left\{\frac{-1+2\lambda_1(1+\phi^2)}{4} <
\lambda_2 \phi \leq \frac{1-2\lambda_1(1+\phi^2)}{4}\,,\ \lambda_1 -
\frac{1}{2(1-\phi)^2} \leq \lambda_2 \leq \frac{1}{2(1+\phi)^2} - \lambda_1
\right\}. \label{D2}
\end{align}

From Corollary \ref{nulleigenvalue} we conclude that $D_{n,\bl}$ has at
least one null eigenvalue inside $\D_\bl$ if $\bl$ belongs to
\begin{equation}
\D_\bl^0 = \bigcup_{\substack{1\,\leq\,k\,\leq\,n \\ n
\in \N}}
\left\{\lambda_1 = \frac{1 + \phi^2 + 4 \phi
\cos \left(\frac{k \pi}{n+1}\right)}{2 \left(1 + \phi^2 + 2 \phi  \cos
\left(\frac{k \pi}{n+1}\right)\right)^2}\,, \ \lambda_2 =
\frac{-\phi}{\left(1+\phi^2 + 2 \phi \cos\left(\frac{k
\pi}{n+1}\right)\right)^2}
\right\}. \label{D0}
\end{equation}
As a result of that, $D_{n,\bl}$ is positive definite, for all $n \in \N$, if
$(\lambda_1, \lambda_2)$ is considered inside $\D_\bl \setminus
\D_\bl^0$, implying that $-\frac{1}{2n} \log(\det(D_{n,\bl}))$ is finite, for
all $n \in \N$. However, we need to be careful when taking the limit as $n \to
\infty$. Although $D_{n,\bl}$ is positive definite in $\overline{\D_\bl^0}
\setminus \D_\bl^0$, asymptotically speaking, the limit $\lim_{n \to \infty}
L_n(\bl)$ does not exist over this set. Consequently, we may define
$\mathcal{L}(\bl) = +\infty$, if $\bl \notin \D_\bl \setminus
\overline{\D_\bl^0}$. Henceforth, we shall restrict our analysis to the
set $\D_\bl \setminus \overline{\D_{\bl}^0}$.

Consider in what follows the measure space $L^\infty(\T) := L^\infty(\T,
\mathcal{B}(\T), \mathbb{L})$, were $\mathbb{L}(\cdot)$ is the Lebesgue measure
acting on $\mathcal{B}(\T)$, the  Borel $\sigma$-algebra over $\T =
[-\pi,\pi)$. Since $\varphi_{\boldsymbol{\lambda}}, \,h_\phi \in L^\infty(\T)$,
it is straightforward to show that
\begin{equation}\label{bounded}
 |\alpha_{n, k}^{\bl}| \leq ||\varphi_{\bl}||_\infty ||h_\phi||_\infty, \quad
\mbox{for all } 1 \leq k \leq n \mbox{ and } n \in \N,
\end{equation}
where $|| \cdot ||_\infty$ denotes the usual norm in $L^\infty(\T)$ (see
definition 6.15 in Bartle \cite{bartle95}). The function
$\varphi_{\bl}\,h_\phi: \T \rightarrow \R$, defined by
 \begin{equation*}
  (\varphi_{\bl}\,h_\phi)(\omega) =
\varphi_{\bl}(\omega)\, h_\phi(\omega) = (\lambda_1+\lambda_2
\cos(\omega))(1+\phi^2 + 2 \phi\cos(\omega)),
 \end{equation*}
is continuous and bounded in $\T$,
hence it attains a maximum and a minimum in that interval. Let
$m_{\varphi_{\bl}
\, h_\phi}$ and $M_{\varphi_{\bl} \, h_\phi}$ denote, respectively, the
\emph{minimum} and the \emph{maximum} of $(\varphi_{\bl}\,h_\phi)(\cdot)$, i.e.,
   \begin{equation*}
    m_{\varphi_{\bl} \, h_\phi} = \min_{\omega \in
\T}\{(\varphi_{\bl} \, h_\phi)(\omega)\} \quad  \mbox{and}
\quad
M_{\varphi_{\bl} \, h_\phi}=\max_{\omega \in
\T}\{(\varphi_{\bl} \, h_\phi)(\omega)\}.
   \end{equation*}
It follows that $m_{\varphi_{\bl} \, h_\phi}$ and $M_{\varphi_{\bl} \,
h_\phi}$ coincide, respectively, with the \emph{essential lower and upper
bounds} of $(\varphi_{\bl}\,h_\phi)(\cdot)$
(see page 65 in Grenander and Szeg\"o \cite{grenander58}). Moreover, one can
verify that
\begin{equation*}
 m_{\varphi_{\bl}\,h_\phi}, M_{\varphi_{\bl}\,h_\phi} \in
\left\{(\lambda_1+\lambda_2)(1+\phi)^2,\ (\lambda_1-\lambda_2)(1-\phi)^2,\
-\frac{\left(\lambda_2 (1+\phi^2)-2 \lambda_1 \phi\right)^2}
{8\lambda_2 \phi} \right\}.
\end{equation*}

Note that
  \begin{align}
   1 - 2\,(\varphi_{\bl}\,h_\phi)(\omega) &= 1 - 2\,(\lambda_1+\lambda_2
\cos(\omega))(1+\phi^2 + 2 \phi\cos(\omega)) \nonumber\\
   &= 1 - 2\lambda_1 (1+\phi^2) - 2\,(2\lambda_1 \phi
+ \lambda_2(1+\phi^2)) \cos(\omega) - 4 \lambda_2 \phi \cos^2(\omega)
\nonumber\\
   &= (p - 2s) + q (2 \cos(\omega)) + s (2 \cos(\omega))^2 = g(2
\cos(\omega)),\label{gcos}
  \end{align}
where $g(\cdot)$ is the second‐order polynomial given in \eqref{polyg}, but
for the particular case when $p,q$ and $s$ are given by \eqref{pqrs}. If
$\bl \in \D_\bl \setminus \overline{\D_{\bl}^0}$, we have $g(x) \geq 0$, for all
$x \in [-2,2]$, and from \eqref{gcos} it follows that
\begin{align*}
1 - 2\,(\varphi_{\bl}\,h_\phi) (\omega) = g(2 \cos(\omega)) \geq 0, \forall\,
\omega \in \T \ \Rightarrow \ (\varphi_{\bl}\,h_\phi)(\omega) \leq 1/2,
\forall\, \omega \in \T.
\end{align*}
Thus, $M_{\varphi_{\bl} \, h_\phi} \leq 1/2$ for all $(\lambda_1,
\lambda_2) \in \D_\bl \setminus \overline{\D_{\bl}^0}$. On the other hand, from
   \begin{equation*}
    ||\varphi_{\bl}||_\infty ||h_\phi||_\infty \geq
||\varphi_{\bl} \, h_\phi||_\infty =
\max\{|M_{\varphi_{\bl}
\, h_\phi}|, |m_{\varphi_{\bl} \, h_\phi}|\} \geq
-m_{\varphi_{\bl} \, h_\phi},
   \end{equation*}
   we obtain $m_{\varphi_{\bl} \, h_\phi} \geq
-||\varphi_{\bl}||_\infty ||h_\phi||_\infty$. Therefore, if $\bl \in \D_\bl
\setminus \overline{\D_{\bl}^0}$, then
\begin{equation}\label{asser}
 [m_{\varphi_{\bl} \, h_\phi}, M_{\varphi_{\bl} \, h_\phi}] \subseteq [-
 ||\varphi_{\bl}||_\infty ||h_\phi||_\infty, 1/2].
\end{equation}

Let us consider the continuous extended function $F:[- ||\varphi_{\bl}||_\infty
||h_\phi||_\infty, 1/2] \rightarrow \R \cup \{\infty\}$, defined by
\begin{equation*}
 F(x) = -\frac{\log(1 - 2 x)}{2}.
\end{equation*}
Note that $F(\cdot)$ has a bounded support (i.e., the set of
those $x \in \R$ for which $F(x) \neq 0$ is bounded) and,  as a consequence from
\eqref{bounded} and \eqref{asser}, if $\bl \in \D_\bl
\setminus \overline{\D_{\bl}^0}$, we infer that $F(\alpha_{n,k}^{\bl})$ are
finite for every $1 \leq k \leq n$ and $n \in \N$. Then, it follows from theorem
5.1 in Tyrtyshnikov \cite{tyrtyshnikov94} that
\begin{equation*}
 \lim_{n \to \infty} \frac{1}{n} \sum_{k=1}^n F(\alpha_{n,k}^{\bl})
= \frac{1}{2\pi} \int_{\T} (F\circ(\varphi_{\bl}\,h_\phi))(\omega)\,d\omega.
   \end{equation*}
In particular, we have
\begin{align}
 & \lim_{n \to \infty} L_n(\lambda_1, \lambda_2)\ =\ \lim_{n \to \infty}
-\frac{1}{2n} \sum_{k=1}^n \log(1 - 2 \alpha_{n,k})\ =\ \lim_{n \to \infty}
\frac{1}{n} \sum_{k=1}^n F(\alpha_{n,k})\ =\ \frac{1}{2\pi} \int_{\T}
(F\circ(\varphi_{\bl}\, h_\phi))(\omega)\, d\omega \nonumber \\
 & = -\frac{1}{4\pi}\int_{-\pi}^\pi \log \Big(1- 2 \,h_\phi(\omega)\,
\varphi_{\bl}(\omega) \Big) d\omega = -\frac{1}{4\pi} \int_{-\pi}^\pi \log
\left[1 - 2\,(1+\phi^2 + 2\phi \cos(\omega)) \left(\lambda_1 + \lambda_2
\cos(\omega)\right) \right] d\omega \nonumber \\
 & \hspace{1.8cm} = -\frac{1}{4\pi} \int_{-\pi}^\pi \log \left[1 - 2\lambda_1
(1+\phi^2) -
2\,\right(2\lambda_1 \phi + \lambda_2(1+\phi^2)\left) \cos(\omega) - 4
\lambda_2 \phi \cos^2(\omega)\right] d\omega
\nonumber\\
 & \hspace{3.6cm} = -\frac{1}{4\pi} \int_{-\pi}^\pi \log \left[p - 2s + 2 q
\cos(\omega) + 4 s \cos^2(\omega) \right] d\omega, \label{inta}
\end{align}
where $p, q$ and $s$ are given by \eqref{pqrs}. Considering $A$ and $B$
as in \eqref{AB}, from Lemma \ref{lemabc} (see Appendix \ref{ApenLemma}) it
follows that
\begin{equation}\label{intpqs}
 \int_{-\pi}^\pi \log \left[p - 2s + 2 q \cos(\omega) + 4 s
 \cos^2(\omega) \right] d\omega  = 2 \pi \log\left[\frac{(p - 2s)(1 + \sqrt{1 -
 A^2})(1 + \sqrt{1 - B^2})}{4}\right].
\end{equation}
In conclusion, \eqref{Llim} now follows from \eqref{inta} and \eqref{intpqs}.
\end{proof}

In Figure \ref{domains}, we ploted the domain $\D_\bl = \D_\bl^1 \cup
\D_\bl^2$, for $\D_\bl^1$ and $\D_\bl^2$ given, respectively, in \eqref{D1} and
\eqref{D2} for $\bl \in[-2,0.5] \times [-3,2]$ and $\phi = 1/3$. In this
figure, we also ploted some of the points $(\lambda_1,\lambda_2)$ that belong
to $\D_\bl^0$, given in \eqref{D0}. Notice how they scatter just over one side
of the boundary of $\D_\bl$.

\begin{figure}[ht!]
   \centering
   \includegraphics[scale=0.7]{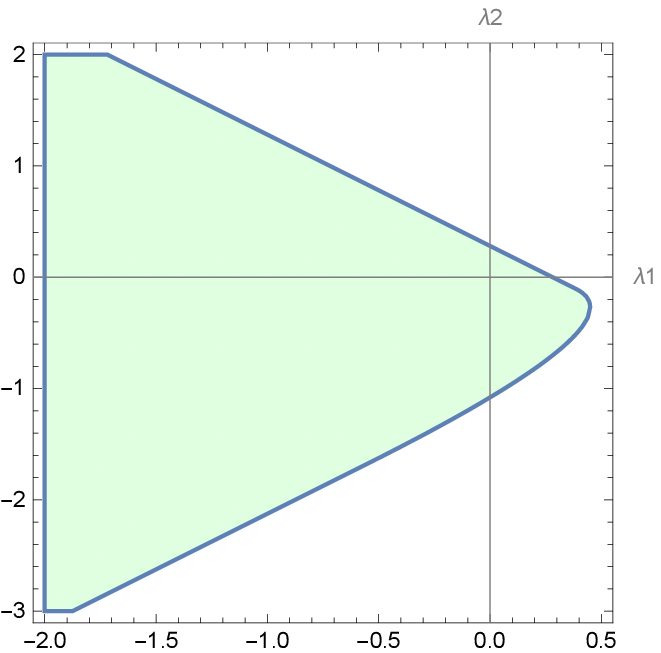} \hspace{1cm}
   \includegraphics[scale=0.6]{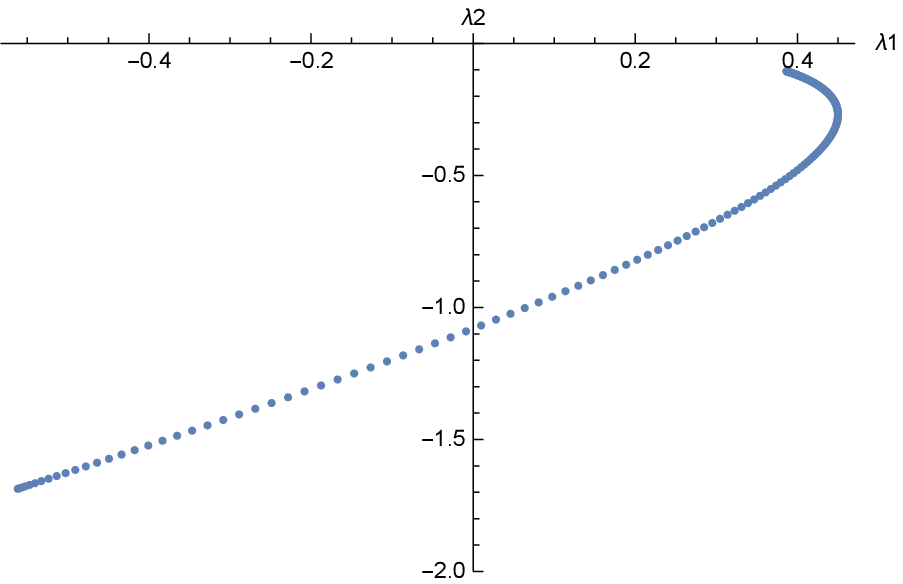}
   \caption{Domain $\D_\bl = \D_\bl^1 \cup \D_\bl^2$, for $\D_\bl^1$ and
$\D_\bl^2$ given, respectively, in \eqref{D1} and \eqref{D2}, for the case
when $\phi = 1/3$, with $(\lambda_1, \lambda_2) \in [-2,0.5] \times [-3,2]$.
On the left is the union $\D_\bl = \D_\bl^1 \cup \D_\bl^2$. On
the right, the points $(\lambda_1,\lambda_2)$ with $\lambda_1
= \frac{1 + \phi^2 + 4 \phi  \cos \left(\frac{\pi  k}{n+1}\right)}{2
\left(1+\phi^2+2 \phi  \cos \left(\frac{\pi  k}{n+1}\right)\right)^2}$ and
$\lambda_2 = \frac{-\phi }{\left(1+\phi^2+2 \phi  \cos \left(\frac{\pi
k}{n+1}\right)\right)^2}$, for $n = 200$ and $1 \leq k \leq 200$, are plotted.
\label{domains}}
\end{figure}



\section{Examples}
Here we introduce three examples that illustrate the theory presented in the
preceding sections. The first one gives a counterexample to show that
proposition 5 in Fasino \cite{fasino88} (and by consequence Lemma \ref{fasino})
is not true if the condition $r \geq p - s$ is not verified.

\begin{example}{}\label{condfasino}
 Consider $p = 5$, $q = -1$, $r = 1$ and $s = 2$, so that $D_n$, defined in
\eqref{matrizdn}, is given by
\begin{equation}\label{Dnexe1}
 D_n = \left[%
 \begin{array}{cccccc}
  1 & -1 & 2 & 0 & \cdots & 0\\
  -1 & 5 & -1 & 2 &\ddots&\vdots \\
  2 & -1 & 5 & \ddots & \ddots & 0 \\
  0 & 2 & \ddots & \ddots & -1 & 2 \\
  \vdots & \ddots & \ddots & -1 & 5 & -1 \\
  0 & \cdots & 0 & 2 & -1 & 1
 \end{array}\right].
\end{equation}
Note that $1 = r < p - s = 2$. Since $q^2 - 4 s (p-2s) = -7$, the polynomial
function $g(x) = s x^2 + q x + p-2s = 2 x^2 -x + 1$ has no real roots. We
observe that, even though $g(x) > 0$ for all $x \in
\R$, the matrix $D_n$ in \eqref{Dnexe1} cannot be non-negative definite for
all $n \in \N$. In fact, when computing its eigenvalues, we observe that
$D_n$ has one negative eigenvalue if $5 \leq n \leq 8$, and two negative
eigenvalues if $9 \leq n \leq 100$, suggesting that proposition 5 in Fasino
\cite{fasino88} does not hold in the absence of the condition $r \geq p - s$.
Nevertheless, it is still possible to compute the determinant of $D_n$ by using
the result of Theorem \ref{thedet}. The coefficients required for this
computation are (in approximated form)
\begin{equation*}
 \kappa_1 = 0.163717 + 0.05368 i, \ \kappa_2 = 0.163717 - 0.05368 i,
\ \kappa_3 = -0.395173,\ \kappa_4 = -0.075118, \ \kappa_5 = 1.14286,
\end{equation*}
and
\begin{equation*}
 \mu_1 = -1.94374 - 0.471031 i,\quad \mu_2 = -1.94374 + 0.471031 i,\quad \mu_3
= 1.03951 , \quad \mu_4 = 3.84797, \quad \mu_5 = 2.
\end{equation*}
Although $\kappa_1, \kappa_2, \mu_1$ and $\mu_2$ are complex numbers, the
determinant is real and it is given by
\begin{equation*}
 \det(D_n) = \sum_{j=1}^5 \kappa_j\,f(\mu_j)\,\mu_j^{n-5},
\end{equation*}
where $f(z) = -z^4+9 z^3-18 z^2-24 z+32$.

Table \ref{numdet1} presents the values of $\det(D_n)$ for four different
values of $n$, obtained with the help of the Wolfram Mathematica
software, operating in an Intel Core i7-8565U processor. For comparison
reasons, when using the determinant function available in this software, the
computational time registered for $n = 2000$ was $0.828125$ seconds. As Table
\ref{numdet1} shows, the formula presented in Theorem \ref{thedet} allows to
compute the determinant of the matrix in \eqref{Dnexe1} much faster than the
usual algorithms do.
\begin{table}[h]
\caption{Approximated values of $\det(D_n)$, when $n \in \{5,\, 5 \times
10^6,\, 5 \times 10^7,\, 5 \times 10^8\}$.}\label{numdet1}
 \begin{center}
 \begin{tabular}{|c||c|c|c|c|}
  \hline \hline
   $n$ & $5$ & $5 \times 10^6$ & $5 \times 10^7$ & $5 \times 10^8$\\
   \hline
   $\det(D_n)$ & $-40$ & $1.65193 \times 10^{2926158}$ & $8.47348 \times
10^{29261604}$ & $1.06844\times 10^{292616072}$ \\
   \hline
   Time (in seconds) & $\approx 0$ & $0.015625$ & $0.046875$ & $0.453125$ \\
   \hline \hline
 \end{tabular}
\end{center}
\end{table}
\fim
\end{example}

The next example clarifies the theory presented in Section \ref{secMA1}.

\begin{example}{}
Consider $\phi = 1/3$ fixed and let $(X_n)_{n \in \N}$ denote the MA(1) process
defined in Section \ref{secMA1}. We demonstrated that the normalized cumulant
generating function associated to the random sequence
$\left(n^{-1}\left(\sum_{k=1}^n X_k^2, \sum_{k=2}^n X_k
X_{k-1}\right)\right)_{n \geq 2}$ can be written as in \eqref{assynlog}. For
instance, if $\bl = (-1,-1)$, then $D_{n,\bl}$ is the pentadiagonal matrix
given by
\begin{equation*}
 D_{n,(-1,-1)} = \left[%
 \begin{array}{cccccc}
  \frac{32}{9} & \frac{16}{9} & \frac{1}{3} & 0 & \cdots & 0\\
  \frac{16}{9} & \frac{35}{9} & \frac{16}{9} & \frac{1}{3}
&\ddots&\vdots \\
  \frac{1}{3} & \frac{16}{9} & \frac{35}{9} & \ddots & \ddots & 0 \\
  0 & \frac{1}{3} & \ddots & \ddots & \frac{16}{9} & \frac{1}{3} \\
  \vdots & \ddots & \ddots & \frac{16}{9} & \frac{35}{9} & \frac{16}{9} \\
  0 & \cdots & 0 & \frac{1}{3} & \frac{16}{9} & \frac{32}{9}
 \end{array}\right].
\end{equation*}
 The vector $(-1,-1)$ belongs to the interior of $\D_\bl^2$, defined
in \eqref{D2}. Hence, from Proposition \ref{propconvLn} we conclude that
\begin{equation*}
 \lim_{n \to \infty} L_n(-1,-1) = \lim_{n \to \infty} -(1/2n)
\log[\det(D_{n,(-1,-1)})] = \mathcal{L}(-1,-1)
\end{equation*}
where
\begin{equation} \label{convexamp}
\mathcal{L}(-1,-1) = -\frac{1}{2} \log\left[\frac{(p-2s)(1 +
\sqrt{1-A^2})(1+\sqrt{1-B^2})}{4}\right] \approx -0.548981,
\end{equation}
with $p = 35/9$, $q = 16/9$, $s = 1/3$, and
\begin{equation*}
A = \frac{q - \sqrt{q^2-4s(p-2s)}}{p-2s} = \frac{16-2i\sqrt{23}}{29}
\quad \mbox{and} \quad
B = \frac{q + \sqrt{q^2-4s(p-2s)}}{p-2s} = \frac{16+2i\sqrt{23}}{29}.
\end{equation*}
Table \ref{numdet} presents the values of $L_n(-1,-1) = -(1/2n)
\log[\det(D_{n,(-1,-1)})]$ for $n \in \{5, 10, 50, 100, 500\}$. Notice that,
even for a small value of $n = 5$, the term $L_n(-1,-1)$ is
relatively close to the asymptotic value in \eqref{convexamp}.

\begin{table}[h]
\caption{Approximated values of $L_n(-1,-1)$, for $n \in \{5, 10, 50, 100,
500\}$.}\label{numdet}
 \begin{center}
 \begin{tabular}{|c||c|c|c|c|c|}
  \hline \hline
   $n$ & $5$ & $10$ & $50$ & $100$ & $500$\\
   \hline
   $L_n(-1,-1)$ & $-0.554116$ & $-0.551548$ & $-0.549495$ &
$-0.549238$ & $-0.549032$ \\
   \hline \hline
 \end{tabular}
\end{center}
\end{table}
\fim
\end{example}

In the following example, we show that in the case when $r = p - s$, the
eigenvalues of the matrix $D_n$ in \eqref{matrizdn} feature a periodic
behavior. This is due to the result of Lemma \ref{eigelema}.

\begin{example}{}
Consider once more the MA(1) process with $\phi = 1/3$. As shown in Lemma
\ref{eigelema}, since the matrix $D_{n,\bl}$ in \eqref{pqrs} satisfies the
relation $r = p - s$, for any pair $\bl \in \R^2$, the eigenvalues of this
matrix are given by $\alpha_{n,k} = 4 s \cos^2\left(\frac{k \pi}{n+1}\right) + 2
q \cos\left(\frac{k \pi}{n+1}\right) + p - 2 s$, for $p, q, r, s$ defined in
\eqref{pqrs} and $1 \leq k \leq n$. If we take a point $\bl$ outside the range
of $\D_\bl$, we shall have an enumerable set of negative eigenvalues of
$D_{n,\bl}$. For instance, if $\bl = (0,1)$, then $p = 1/3$, $q = -10/3$, $r =
2/3$ and $s = -1/3$. The matrix $D_{n,(0,1)}$ is therefore given by
\begin{equation*}
 D_{n,(0,1)} = \left[%
 \begin{array}{cccccc}
  \frac{2}{3} & -\frac{10}{3} & -\frac{1}{3} & 0 & \cdots & 0\\
  -\frac{10}{3} & \frac{1}{3} & -\frac{10}{3} & -\frac{1}{3}
&\ddots&\vdots \\
  -\frac{1}{3} & -\frac{10}{3} & \frac{1}{3} & \ddots & \ddots & 0 \\
  0 & -\frac{1}{3} & \ddots & \ddots & -\frac{10}{3} & -\frac{1}{3} \\
  \vdots & \ddots & \ddots & -\frac{10}{3} & \frac{1}{3} & -\frac{10}{3} \\
  0 & \cdots & 0 & -\frac{1}{3} & -\frac{10}{3} & \frac{2}{3}
 \end{array}\right].
\end{equation*}
If $n = 5$, the eigenvalues of $D_{n,(0,1)}$ are
\begin{equation}\label{eige1}
 \alpha_{5,1} = -\frac{10}{3 \sqrt{3}}, \quad \alpha_{5,2} = -\frac{4}{9},
\quad \alpha_{5,3} = 1, \quad \alpha_{5,4} = \frac{16}{9}, \quad \alpha_{5,5} =
\frac{10}{3 \sqrt{3}}.
\end{equation}
If we take $n = 11$, these same eigenvalues will appear as
\begin{equation*}
 \alpha_{11,2} = -\frac{10}{3 \sqrt{3}}, \quad \alpha_{11,4} = -\frac{4}{9},
\quad \alpha_{11,6} = 1, \quad \alpha_{11,8} = \frac{16}{9}, \quad
\alpha_{11,10} = \frac{10}{3 \sqrt{3}}.
\end{equation*}
In fact, if $k \equiv 5\,(\mbox{mod } 6)$, then the values in \eqref{eige1}
will be eigenvalues of $D_{k,(0,1)}$. Consequently, since $\alpha_{5,1}$ and
$\alpha_{5,2}$ are already negative, $D_{k,(0,1)}$ cannot be
non-negative definite.

This reasoning is not restricted to the pentadiagonal matrix associated to the
MA(1) process. If $D_n$ in \eqref{matrizdn} has arbitrary values for $p, q, s$,
and $r$ is such that $r = p - s$, then its eigenvalues also share
this periodic property, due to Lemma \ref{eigelema}. The point here is that the
presence of periodic eigenvalues does not allow the existence of some $n_0 \in
\N$ such that $D_n$ is positive or non-negative definite for $n \geq n_0$.

\fim
\end{example}


\section{Conclusions}\label{conclusionsection}

In this work, we have examined some determinantal properties of the general
matrix $D_n$ in \eqref{matrizdn}. We gave explicit expressions for the
determinant via recurrence relations, providing an alternative to
the existing expressions given in the literature. Theorem \ref{thedet}, with
the help of five lemmas, presents the explicit expression for the determinant
of $D_n$, showing its dependence on $n \geq 3$. We also analyzed when the
matrix $D_n$ is non-negative definite in the presence of the restriction $r
\geq p - s$. This is achieved through the use  of proposition 5
in Fasino \cite{fasino88}, that helped us to provide the precise expressions
for the domains in \eqref{D1234}. Furthermore, when $r = p - s$, we showed the
explicit domain in which $D_n$ is actually positive definite (in the strict
sense). Example \ref{condfasino} is important to show that the condition
$r \geq p - s$ is essential for proposition 5 in Fasino \cite{fasino88}.

We have indicated the importance of the linear
algebra analysis, self-contained in the present work, by applying
these results to the stationary centered moving average process of first
order with Gaussian innovations.  An explicit expression for
the normalized cumulant generating function $L_n(\cdot)$, associated to $\W_n$,
described in expression \eqref{Wn}, was exhibited. Proposition \ref{propconvLn}
presents the limit of this function $L_n(\cdot)$, when $n$ goes to infinity, in
a closed expression, whenever it is well defined.  In Example 5.2, with the help
of Proposition \ref{propconvLn}, we expressed the value of the limit
$\mathcal{L}(\cdot)$ in a particular case. Whereas in Example 5.3 we exhibit
the case $r = p\,–\,s$, where the eigenvalues of the matrix $D_n$ in \eqref{matrizdn} feature a periodic behavior, due to Lemma \ref{eigelema}. Finally, we mention that one can obtain the  partial derivatives  of $\mathcal{L}(\bl)$ with respect to $\bl$ by using the Wolfram Mathematica software. From this, one can access an explicit form of the moments for
the underlying random process.

\appendix
\section{A Useful Lemma}\label{ApenLemma}

In this section, we show a useful lemma that makes it possible to compute the
integral in \eqref{intpqs} and which extends the result given in equation
4.224(9) in Gradshteyn and Ryzhik \cite{gradshteyn07}.

 \begin{lemma}\label{lemabc}
  Consider $a, b, c \in \R$ such that $a + b x + c x^2 \geq 0$, for $|x|
\leq 1$. Let
  \begin{equation*}
   \gamma_1 = \frac{b - \sqrt{b^2 - 4 a c}}{2 a} \quad \mbox{and} \quad
\gamma_2 = \frac{b + \sqrt{b^2 - 4 a c}}{2 a}.
  \end{equation*}
  Then, it follows that
  \begin{equation}\label{logabc}
   \int_{0}^\pi \log \left[a + b \cos(\omega) + c \cos^2(\omega)\right] d\omega
= \begin{cases}
   \pi \log\left(\frac{c}{4}\right)& \mbox{if }\ a = b = 0,\ c > 0,\\[2mm]
   \pi \log\left[\frac{a\,(1+\sqrt{1-\gamma_1^2})
(1+\sqrt{1-\gamma_2^2})}{4}\right],& \mbox{otherwise}.
  \end{cases}
  \end{equation}
 \end{lemma}
 \begin{proof}
  Note that, the assumption $a + b x + c x^2 \geq 0$ for $|x|
\leq 1$, guarantees that $a + b \cos(\omega) + c \cos^2(\omega) \geq 0$ for all
$\omega \in [0, \pi]$, since $|\cos(\omega)| \leq 1$. In particular, it follows
that $a \geq 0$, and for this reason, two cases are considered for the proof.
\begin{itemize}
   \item {\it Case 1:} If $a = 0$, then we must have, necessarily, $b = 0$ and
$c > 0$. Hence
  \begin{align*}
   \int_{0}^\pi \log \left[a + b \cos(\omega) + c \cos^2(\omega)\right] d\omega
   &= \int_{0}^\pi \log \left[c \cos^2(\omega)\right] d\omega = \int_{0}^\pi
\log \,(c)\, d\omega + \int_{0}^\pi \log \left[\cos^2(\omega)\right] d\omega\\
   &= \pi \log\,(c) + 4 \int_{0}^{\pi/2} \log \left[\cos(\omega)\right]
d\omega.
  \end{align*}
    From equation 4.224(6) in Gradshteyn and Ryzhik \cite{gradshteyn07}, we
know
that $\int_{0}^{\pi/2} \log \left[\cos(\omega)\right] d\omega = -\frac{\pi}{2}
\log\,(2)$, hence
  \begin{align*}
   \int_{0}^\pi \log \left[c \cos^2(\omega)\right] d\omega
   &= \pi \log\,(c) - \pi \log\,(4) = \pi \log\left(\frac{c}{4}\right).
  \end{align*}

  \item {\it Case 2:} If $a > 0$, note that
  \begin{equation}
   a + b \cos(\omega) + c \cos^2(\omega) = a\,(1 + \gamma_1 \cos(\omega))(1
+ \gamma_2 \cos(\omega)).
  \end{equation}
  The proof of \eqref{logabc} then follows from the identify
  \begin{equation*}
   \int_{0}^\pi \log \left[1 + z \cos(\omega)\right] d\omega =
\pi \log\left[\frac{1 + \sqrt{1 - z^2}}{2}\right], \quad \mbox{for }\ z \in
\mathbb{C}.
  \end{equation*}
  Indeed, we have
  \begin{align*}
   &\int_{0}^\pi \log \left[a + b \cos(\omega) + c \cos^2(\omega)\right]
d\omega = \int_{0}^\pi \log \left[a\,(1 + \gamma_1 \cos(\omega))(1
+ \gamma_2 \cos(\omega))\right] d\omega \\
   &= \pi \log(a) \, + \int_{0}^\pi \log \left[1 + \gamma_1 \cos(\omega)\right]
d\omega \, + \int_{0}^\pi \log \left[1 + \gamma_2 \cos(\omega)\right] d\omega\\
   &= \pi \log(a) \, + \pi \log\left[\frac{1 + \sqrt{1 -
\gamma_1^2}}{2}\right]+\pi \log\left[\frac{1 + \sqrt{1 -
\gamma_2^2}}{2}\right]\\
   &= \pi \log\left[\frac{a\,(1+\sqrt{1-\gamma_1^2})
(1+\sqrt{1-\gamma_2^2})}{4}\right].
  \end{align*}
 \end{itemize} \vspace{-10mm} \end{proof}

\section*{Acknowledgments}

M.J.\ Karling was supported by Coordena\c{c}\~{a}o de Aperfei\c{c}oamento de
Pessoal de N\'{i}vel Superior (CAPES)-Brazil (1736629) and Conselho Nacional
de Desenvolvimento Cient\'{i}fico e Tecnol\'{o}gico (CNPq)-Brazil
(170168/2018-2). A.O.\ Lopes' research was partially supported by CNPq-Brazil
(304048/2016-0). S.R.C.\ Lopes' research was partially supported by
CNPq-Brazil (303453/2018-4). The authors would like to thank  H.H.\ Ferreira
for supplying us with an automated form of the matrix $D_n$ in Wolfram
Mathematica.

\section*{Declaration of competing interest}

None declared.


\end{document}